\documentclass[a4paper,twoside,fleqn]{article}
\usepackage{CJK} 
\usepackage{amscd}
\usepackage{amsfonts}
\usepackage[dvips]{graphics}
\usepackage{amssymb}
\usepackage[reqno]{amsmath}
\usepackage{amsthm}
\usepackage{makeidx}
\usepackage{color}
\usepackage{graphicx}
\usepackage{subfigure}
\usepackage{multicol}
\usepackage{fancyhdr}
\usepackage[colorlinks,linkcolor=blue,anchorcolor=blue,citecolor=blue]{hyperref}
\usepackage[numbers,sort&compress]{natbib}

 \newtheorem{theorem}{\sc\bf Theorem}[section]
 
 \newtheorem{corollary}[theorem]{\sc\bf Corollary}
 \newtheorem{lemma}[theorem]{\sc\bf Lemma}
 \newtheorem{proposition}[theorem]{\sc\bf Proposition}
 \newtheorem{definition}[theorem]{\sc\bf Definition}
 \newtheorem{remark}[theorem]{\sc\bf Remark}

  \numberwithin{equation}{section}

\usepackage[numbers]{natbib}
\setcitestyle{open={},close={}}

\large

\title{{\bf B-Fredholm theory in Banach algebras}}

\author{Yunnan \textsc{Zhang}$^{\small \makebox{a}}$, Qingping \textsc{Zeng}$^{\small \makebox{b},}$\thanks{Corresponding author. Email address: zqpping2003@163.com.} \ \ and Zhenying \textsc{Wu}$^{\small \makebox{a}}$ \\
\small a. School of Mathematics and Statistics, Fujian Normal University, Fuzhou 350117, P.R. China \\
\small b. College of Computer and Information Sciences, Fujian Agriculture and Forestry University, \\ \small  Fuzhou 350002, P.R. China  }

\begin{document}

\date{}
\maketitle

\large

\begin{quote}
 {\bf Abstract:} ~ The aim of this paper is to develop a systematic B-Fredholm theory in semiprime Banach algebras. We first generalize Smyth's important punctured neighbourhood theorem to B-Fredholm elements. Then using this result, we investigate the local spectral theory of B-Fredholm elements, including the localized left (resp. right) SVEP and a classification of components of B-Fredholm resolvent set. Finally, in semisimple Banach algebra context, we characterize element $f$ such that $f^{n}$ belongs to the socle for some $n \in \mathbb{N}$ from two different perspectives: one is the invariance of the B-Fredholm spectrum under commuting perturbation $f$, the other is the Rieszness and the B-Fredholmness of $f$.
\end{quote}

\section{Introduction} \label{sec1}

It is Atkinson's theorem ([\cite{Barnes-Murphy-Smyth-West}, Theorem O.2.2]) that the set of Fredholm operators on a Banach space $X$ can be characterized as those bounded linear operators invertible modulo the finite rank ideal $F(X)$. It follows from this characterization that Fredholm operators on Banach spaces has a natural extension to the more general setting of Banach algebras, by replacing the ideal $F(X)$ with the ideal $soc(\mathcal{A})$, the socle of a Banach algebra $\mathcal{A}$. Fredholm theory in Banach algebras was pioneered by B.A. Barnes [\cite{Barnes-Fredholm-element,Barnes-Fredholm-theory}], and was further developed by M.R.F. Smyth in [\cite{Smyth-Fredholm-theory}],
see also the monograph [\cite{Barnes-Murphy-Smyth-West,Aiena-book}] and the references [\cite{Grobler-Raubenheimer,Grobler-Raubenheimer-II,Smyth,Harte-Fredholm-theory,Mouton-Raubenheimer,Pearlman-Riesz-points}], etc.

In [\cite{Berkani-BFredholm-operators}], M. Berkani introduced the class of B-Fredholm operators, which contains the class of Fredholm operators as a proper subclass, and an Atkinson type characterization for these operators was obtained in [\cite{Berkani-Sarih-BFredholm-operators}]: $T$ is a B-Fredholm operator on a Banach space $X$ if and only if $T$ is Drazin invertible modulo the finite rank ideal $F(X)$. This characterization also leads to a natural definition of B-Fredholm elements in Banach algebras. Basic properties and the index of this class of elements were firstly investigated in [\cite{Berkani-BFredholm,Berkani-rings}].

In this paper, we are aimed to develop a systematic B-Fredholm theory in semiprime Banach algebras. In Section 2, Smyth's punctured neighbourhood theorem is generalized to B-Fredhollm elements. It plays a central role in our investigations. The subsequent two sections address the local spectral theory of B-Fredholm elements.
In Section 3, we characterize the left and right single-valued extension property at $\lambda_{0} \in \mathbb{C}$ for $x \in \mathcal{A}$ in the case that $\lambda_{0}-x$ is a B-Freholm element. Then using these equivalences, in Section 4, we obtain a classification of components of B-Fredholm resolvent set. We also give some interesting applications of the classification. In particular, we can see that the elements having empty B-Fredholm spectrum are exactly those algebraic elements, i.e., the elements that satisfy a non-trivial polynomial identity. In Section 5, we show that the B-Fredholm spectrum is invariant under any commuting perturbation $f$ such that $f^{n} \in soc(\mathcal{A})$ for some $n \in \mathbb{N}$, and conversely this perturbation property characterizes such elements $f$ in the case that $\mathcal{A}$ is semisimple, by using the characterization of algebraic elements and some techniques developed in [\cite{Haily-Kaidi-Palacios}].
In the last section, we characterize such elements $f$ from a different perspective. In particular, we prove that the class of such elements is exactly the intersection of the class of Riesz elements and the class of B-Fredholm elements.

These results generalize the corresponding ones in Banach spaces, using different techniques. Due to the lack of underlying Banach space $X$, the spectral theory, including B-Fredholm theory, in Banach algebras is more difficult than that in Banach spaces. In the coming up manuscripts, we will develop a systematic spectral theory in Banach algebras, basing on the results obtained in the present paper.

An algebra $\mathcal{A}$ is said to be semiprime if $\{0\}$ is the only two-sided ideal $J$ for which $J^{2}=\{0\}$. Throughout this paper, we always assume that $\mathcal{A}$ is a semiprime, complex and unital Banach algebra, unless otherwise specified.

\section{The punctured neighbourhood theorem for B-Fredhollm elements} \label{sec2}

 A non-zero idempotent $e \in \mathcal{A}$ is minimal if $e\mathcal{A}e$ is a division algebra.
Let $Min(\mathcal{A})$ denote the set of minimal idempotents of $\mathcal{A}$. It is well known that $I$ is a minimal left (resp. right) ideal if and only if $I=\mathcal{A}e$ (resp. $I=e\mathcal{A}$) for some $e \in Min(\mathcal{A})$ (see [\cite{Bonsall-Duncan}, Proposition 30.6]). The following concept is important to develop Fredholm theory in Banach algebra.

\begin {definition} ${\label{2.1}}$  \begin{upshape}(see [\cite{Barnes-Fredholm-theory}])
A right (resp. left) ideal $J$ of  $\mathcal{A}$ is said to be of finite order if $J$ can be written as the sum of a finite number of minimal right (resp. left) ideals of $\mathcal{A}$. The order $\Theta(J)$ of $J$ is defined as the smallest number of minimal right (left) ideals which have sum $J$. By convention, $\Theta(\{0\})=0$ and $\Theta(J)=\infty$ if $J$ does not have finite order.
 \end{upshape}
\end {definition}

The socle of $\mathcal{A}$, $soc(\mathcal{A})$ is defined as the sum of the minimal right ideals (which equals to the sum of the minimal left ideals) or $\{0\}$ if there are none minimal right ideals. When $\mathcal{A}$ is semiprime, $soc(\mathcal{A})$ always exists (see [\cite{Bonsall-Duncan}, Proposition 30.10]).

\begin {lemma} ${\label{2.2}}$  \begin{upshape}  (see [\cite{Barnes-Fredholm-theory,Smyth-Fredholm-theory}])
Let $J$ and $K$ be right (left) ideals of $\mathcal{A}$.

$(1)$\,\, $\Theta(J)=n$ if and only if there exist orthogonal minimal idempotents $e_{1},\cdots,e_{n}$ such that
$J=e_{1}\mathcal{A}\oplus \cdots \oplus e_{n}\mathcal{A}$ ($J=\mathcal{A}e_{1}\oplus \cdots \oplus\mathcal{A}e_{n}$).

$(2)$\,\, If $\Theta(K)<\infty$ and $J$ is properly contained in $K$, then $J$ has finite order and $\Theta(J)<\Theta(K)$.

$(3)$\,\, $\Theta(x\mathcal{A})=\Theta(\mathcal{A}x)$ for every $x \in \mathcal{A}$.

$(4)$\,\, $soc(\mathcal{A})=\{x \in \mathcal{A}: \Theta(x\mathcal{A}) <\infty \}$.

$(5)$\,\, $J \subseteq soc(\mathcal{A})$ if and only if $\Theta(J) <\infty \}$.
\end{upshape}
\end {lemma}

For $x\in \mathcal{A}$, the right annihilator of $x$ in $\mathcal{A}$ is defined by
$$R(x)=\{a \in \mathcal{A}: xa=0\},$$
while the left annihilator of $x$ in $\mathcal{A}$ is defined by
$$L(x)=\{a\in \mathcal{A}:ax=0\}.$$

\begin {definition} ${\label{2.3}}$  \begin{upshape}
For $x\in \mathcal{A}$, the nullity and defect of $x$ are defined by $null(x)=\Theta(R(x))$ and $def(x)=\Theta(L(x))$ respectively.
 \end{upshape}
\end {definition}

Let $\mathcal{B}(X)$ denote the Banach algebra of all bounded linear operators on a Banach space $X$. For $T \in \mathcal{B}(X)$, the nullity and defect of $T$ as an operator are defined as $n(T)=\dim ker(T)$ and $d(T)=\dim X/ran(T)$, where $ker(T)$ and $ran(T)$ are the kernel and range of $T$, respectively. For left or right Fredholm operator $T$, the nullity (resp. defect) of $T$ as an element equals to that of $T$ as an operator:

\begin {proposition} ${\label{2.4}}$  \begin{upshape}
Let $T \in \mathcal{B}(X)$ be left or right Fredholm. Then $null(T)=n(T)$ and $def(T)=d(T)$.
 \end{upshape}
\end {proposition}

\begin{proof} Let $T$ be left Fredholm. Then there exist $S \in \mathcal{B}(X)$ and $P \in soc(\mathcal{B}(X))=F(X)$ such that
$ST=I-P$ and the rank $rank(P)$ of $P$ equals to $n(T)$, where $F(X)$ denotes the ideal of finite rank operators on $X$. Observe that $R(T)=P\mathcal{B}(X)$.
It follows that $null(T)=rank(P)=n(T)$. A similar proof shows that if $T$ is right Fredholm, then $def(T)=d(T)$.

In the case $T$ is left (right) Fredholm  but not Fredholm, we have $def(T)=d(T)=\infty$ $(null(T)=n(T)=\infty)$.
\end{proof}

\begin {definition} ${\label{2.5-a}}$ \begin{upshape} (see [\cite{Barnes-Fredholm-element}, Definition 2.1]) An element $a \in \mathcal{A}$ is called Fredholm if $a$ is invertible modulo $\mathrm{soc}(A)$.
 \end{upshape}
\end {definition}

Recall that an element $a$ in a ring $\mathcal{R}$ is called Drazin invertible if there exists $b \in \mathcal{R}$
such that
$$bab=b, ab=ba \makebox{ and } a^{k}ba=a^{k}$$
for some $k \in \mathbb{N}$. In this case, $b$ is called the Drazin inverse of $a$.
If the Drazin inverse of $a$ exists, it is unique and belongs to the double commutant of $a$.
The Drazin index of $a$ is the least non-negative integer $k$ for which the above equations hold.

\begin {definition} ${\label{2.5}}$ \begin{upshape} (see [\cite{Berkani-BFredholm}, Definition 1.1]) An element $a \in \mathcal{A}$ is called B-Fredholm if $\pi(a)$ is Drazin invertible in the quotient algebra $\mathcal{A}/ \mathrm{soc}(A)$, where $\pi:\mathcal{A}\rightarrow \mathcal{A}/ \mathrm{soc}(A)$ is the canonical homomorphism.
 \end{upshape}
\end {definition}

 In the case of $soc(\mathcal{A})=\{0\}$, the B-Fredholm elements in $\mathcal{A}$ are exactly the Drazin invertible elements in $\mathcal{A}$. For this reason, from now on we always assume that $soc(\mathcal{A})$ is not reduced to $\{0\}$.

Denoted by $B\Phi(\mathcal{A})$ the set of all B-Fredholm elements in $\mathcal{A}$. Recall that an element $a \in \mathcal{A}$ is relatively regular if $aba=a$ for some $b \in \mathcal{A}$. In this case $b$ is called an inner inverse of $a$.
If $a \in \mathcal{A}$ is a relatively regular element (with an inner inverse $b$), then $p:=ab$ is an idempotent satisfying $a \mathcal{A}=p\mathcal{A}$, thus $a \mathcal{A}$ is closed.

In the following, we give an improvement of Smyth's punctured neighbourhood theorem [\cite{Smyth-Fredholm-theory}, Theorem 4.6]. This result is crucial in the B-Fredholm theory.

\begin {theorem} ${\label{2.6}}$  \begin{upshape}
Let $x \in B\Phi(\mathcal{A})$. Then there exists $\varepsilon >0$ such that for $0<|\lambda|<\varepsilon$ and sufficiently large $m \in \mathbb{N}$,

$(1)$ $x-\lambda$ is Fredholm.

$(2)$ $null(x-\lambda)$ equals to the constant $\Theta(R(x)\cap x^{m}\mathcal{A})\leq null(x)$.

$(3)$ $def(x-\lambda)$ equals to the constant $\Theta(L(x)\cap \mathcal{A}x^{m}) \leq def(x)$.
 \end{upshape}
\end {theorem}

\begin{proof} (1) By [\cite{Berkani-BFredholm}, Theorem 3.1], there exists $\delta >0$ such that $x-\lambda$ is Fredholm, for $0<|\lambda|<\delta$.

(2) Since $x$ is B-redholm, $x^{n}$ is generalized Fredholm for some $n \in \mathbb{N}$ (see [\cite{Berkani-rings}, Theorem 2.9]), in the sense that there exists $y \in \mathcal{A}$ with
$$ x^{n}yx^{n}-x^{n} \in soc(\mathcal{A}) \,\, \makebox{and} \,\, 1-x^{n}y-yx^{n} \in \Phi(\mathcal{A}).$$
By [\cite{Aupetit-Mouton}, Corollary 2.10], $(x^{n}yx^{n}-x^{n})r(x^{n}yx^{n}-x^{n})=x^{n}yx^{n}-x^{n}$ for some $r \in soc(\mathcal{A})$. Set $y_{0}=y-r+yx^{n}r+rx^{n}y+yx^{n}rx^{n}y$.
Then $x^{n}y_{0}x^{n}=x^{n}$ and $\pi(1-x^{n}y_{0}-y_{0}x^{n})=\pi(1-x^{n}y-yx^{n})$, thus $s:=1-x^{n}y_{0}-y_{0}x^{n} \in \Phi(\mathcal{A})$.

{\bf Claim 1:} $R(x) \cap x^{n} \mathcal{A} \subseteq R(s)$. Indeed, for $z \in R(x) \cap x^{n}\mathcal{A} \subseteq  R(x^{n}) \cap x^{n}\mathcal{A}$, we have
$z=(1-y_{0}x^{n})z= x^{n}y_{0}z$, and hence $sz=(1-x^{n}y_{0}-y_{0}x^{n})z=0$. Consequently, $R(x) \cap x^{n} \mathcal{A} \subseteq R(s)$.

Since $x^{n}$ is generalized Fredholm,  $x^{nm}$ is also generalized Fredholm, hence $x^{nm}$ is relatively regular for each $m \in \mathbb{N}$. Keeping in mind the fact we recalled proceeding this theorem, we get $x^{nm} \mathcal{A}$ is closed. Let $M:=\bigcap\limits_{k=1}^{\infty} x^{k} \mathcal{A}$. Clearly, $M=\bigcap\limits_{m=1}^{\infty} x^{nm} \mathcal{A}$ is closed.

{\bf Claim 2:} $xM=M$. Indeed, $xM \subseteq M$ is trivial. Because $\{R(x) \cap x^{m} \mathcal{A} \}_{m=n}^{\infty}$ is a decreasing sequence of right ideals of finite order, we can choose an integer $m \geq n$ such that $R(x) \cap x^{m} \mathcal{A}= R(x) \cap M$ by Lemma \ref{2.2}(2). Let $y \in M$. Then there exists $\{a_{k}\}_{k=1}^{\infty}$ such that $y=x^{m+k}a_{k}$. Set $z_{k}=x^{m}a_{1}-x^{m+k-1}a_{k}$ for all $k \in \mathbb{N}$. Then $xz_{k}=0$ and so
$z_{k} \in R(x) \cap x^{m}\mathcal{A} = R(x) \cap M$. Therefore, $x^{m}a_{1}=z_{k}+x^{m+k-1}a_{k} \in x^{m+k-1}\mathcal{A}$ for all $k \in \mathbb{N}$. Consequently, $y=x(x^{m}a_{1}) \in xM$.

Since $R(x) \cap M \subseteq R(s)$, $R(x) \cap M$ is a right ideal of finite order, and hence we can find some idempotent $p \in soc(\mathcal{A})$ such that
$R(x) \cap M=p\mathcal{A}$. Define $\hat{x}: (1-p)M \rightarrow M$ by $\hat{x}(a)=xa$ for all $a \in (1-p)M$. Then $\hat{x}$ is surjective and
$$ker(\hat{x})=R(x) \cap (1-p)M=R(x) \cap M \cap (1-p)M \subseteq p\mathcal{A} \cap (1-p)\mathcal{A}=\{0\}.$$
That is $\hat{x}: (1-p)M \rightarrow M$ is invertible. Let $\hat{x}^{-1}:  M \rightarrow (1-p)M$ be the inverse of $\hat{x}$ and $j: (1-p)M \rightarrow M$ be the embedding map. Take $\varepsilon = \min\{ \delta, \frac{1}{2}||\hat{x}^{-1}||^{-1} \}$.

{\bf Claim 3:} $null(x-\lambda) = \Theta(p\mathcal{A})$ for $0<|\lambda|<\varepsilon$. Let $y \in R(x-\lambda) \cap (1-p)\mathcal{A}$. Since $R(x-\lambda) \subseteq M$, $y=(1-p)y \in (1-p)M$. This shows that $$R(x-\lambda) \cap (1-p)\mathcal{A}=R(x-\lambda) \cap (1-p)M.$$
Now let $z \in R(x-\lambda) \cap (1-p)M$. Then $||z||=||\hat{x}^{-1}\hat{x}z||\leq ||\hat{x}^{-1}||\cdot||xz||$, and thus $||(x-\lambda)z||\geq (||\hat{x}^{-1}||^{-1}-|\lambda|)||z||\geq \varepsilon ||z||$, which implies $z=0$.
Therefore $$R(x-\lambda) \cap (1-p)\mathcal{A}= \{ 0 \}.$$
As $\mathcal{A}=p\mathcal{A}\oplus(1-p)\mathcal{A}$, we infer by [\cite{Barnes-Fredholm-element}, Lemma 1.2] that
$$null(x-\lambda) \leq \Theta(p\mathcal{A}).$$
Let $m \in p\mathcal{A}=R(x) \cap M$. Then $(x-\lambda)(1-\lambda j \hat{x}^{-1})^{-1}m=xm=0$. Therefore,
$(1-\lambda j \hat{x}^{-1})^{-1}p\mathcal{A} \subseteq R(x-\lambda).$
Since $p\hat{x}^{-1}=0$, we obtain $p(1-\lambda j \hat{x}^{-1})^{-1}p\mathcal{A}=p\mathcal{A}$.
Note that, since $x-\lambda$ is Fredholm, $R(x-\lambda)=p_{\lambda}\mathcal{A}$ for some idempotents $p_{\lambda} \in soc(\mathcal{A})$.
Consequently, $p\mathcal{A} \subseteq pR(x-\lambda)=pp_{\lambda}\mathcal{A}$. By Lemma \ref{2.2}(2) and (3), it follows that
 $$\Theta(p\mathcal{A})\leq \Theta(pp_{\lambda}\mathcal{A})=\Theta(\mathcal{A}pp_{\lambda})\leq \Theta(\mathcal{A}p_{\lambda})=\Theta(p_{\lambda}\mathcal{A})=null(x-\lambda).$$

 (3) The proof is similar to that of (2), we omit it here.
\end{proof}

\section{SVEP for B-Fredhollm elements} \label{sec3}

For the convenience of the reader we recall some notations for bounded linear operators.  Associated with $T \in \mathcal{B}(X)$, some important invariant subspaces (not necessarily closed) of $T$ are the hyperrange $\bigcap\limits_{n=1}^{\infty} ran(T^{n})$ of $T$, the hyperkernel $\bigcup \limits_{n=1}^{\infty} ker(T^{n})$ of $T$, the analytical core of $T$ defined by
$$K(T):=\{x \in X: \makebox{there exist a sequence } \{x_{n}\}_{n=1}^{\infty} \makebox{ in } X \makebox{ and a constant } \delta >0 $$
 $$\makebox{ such that } Tx_{1}=x, Tx_{n+1}=x_{n} \makebox{ and } ||x_{n}||\leq \delta^{n}||x|| \makebox{ for all } n \in \mathbb{N} \},$$
and the quasinilpotent part of $T$ defined by $H_{0}(T):=\{ x \in X: \lim\limits_{n \rightarrow \infty} ||T^{n}x||^{\frac{1}{n}}=0 \}$.
These subspaces were intensively investigated and turned out to have an important role in local spectral theory and Fredholm theory,
see the monograph [\cite{Aiena-book}] by Aiena.

Another important property in local spectral theory is the so called single-valued extension property, which was firstly introduced by Dunford in [\cite{Dunford-1,Dunford-2}].
An operator $T \in \mathcal{B}(X)$ is said to have the single-valued extension property at $\lambda \in \mathbb{C}$
(SVEP at $\lambda$ for the sake of convenience), if for every neighbourhood $U$ of $\lambda$ the only holomorphic function
 $f: U \rightarrow X$ which satisfies the equation $(\mu-T)f(\mu)=0$ on $U$  is the constant function $f \equiv 0$. The localized SVEP at a point was introduced  by Finch in [\cite{Finch}].

 In the following, we introduce the corresponding concepts for Banach algebra elements.

\begin {definition} ${\label{3.1}}$  \begin{upshape}
An element $x \in \mathcal{A}$ is said to have the left single-valued extension property at $\lambda \in \mathbb{C}$
(left SVEP at $\lambda$ for the sake of convenience), if for every neighbourhood $U$ of $\lambda$ the only holomorphic function
 $f: U \rightarrow \mathcal{A}$ which satisfies the equation $(\mu-x)f(\mu)=0$ on $U$  is the constant function $f \equiv 0$.

 Dually, we shall say that $x \in \mathcal{A}$ have the right single-valued extension property at $\lambda \in \mathbb{C}$
(right SVEP at $\lambda$ for the sake of convenience), if for every neighbourhood $U$ of $\lambda$ the only holomorphic function
 $f: U \rightarrow \mathcal{A}$ which satisfies the equation $f(\mu)(\mu-x)=0$ on $U$  is the constant function $f \equiv 0$.

 An element $x \in \mathcal{A}$ is said to have the left (resp. right) SVEP if $x$ has the left (resp. right) SVEP at every $\lambda \in \mathbb{C}$.
 \end{upshape}
\end {definition}

For $x \in \mathcal{A}$, let $L_{x}$ and $R_{x}$ denote the left and right multiplication operators of $x$ on $\mathcal{A}$. That is,
$$L_{x}(a)=xa \makebox{ and } R_{x}(a)=ax, \makebox{ for all } a \in \mathcal{A}.$$

\begin {remark} ${\label{3.1-a}}$  \begin{upshape}
(1) It is clear that $x \in \mathcal{A}$ has the left (resp. right) SVEP at $\lambda$ if and only if $L_{x}$ (resp. $R_{x}$) has SVEP at $\lambda$.

(2) It is worth to mention that $T \in \mathcal{B}(X)$ has SVEP at $\lambda$ if and only if $L_{T}$ has the left SVEP at $\lambda$;
$T^{*}$ has SVEP at $\lambda$ if and only if $R_{T}$ has the left SVEP at $\lambda$. This result is due to G\^{i}ndac [\cite{Gindac}].
\end{upshape}
\end {remark}

We also define the left hyperrange, the left hyperkernel, the left analytical core $K_{l}(x)$ and the left quasinilpotent part $H_{l}(x)$ of $x \in \mathcal{A}$ exactly as the hyperrange, the hyperkernel, the analytical core and the quasinilpotent part of the left multiplication operator $L_{x}$, respectively.
Similarly, the right hyperrange, the right hyperkernel, the right analytical core $K_{r}(x)$ and the right quasinilpotent part $H_{r}(x)$ of $x \in \mathcal{A}$ can be defined exactly as the hyperrange, the hyperkernel, the analytical core and the quasinilpotent part of the right multiplication operator $R_{x}$, respectively.

Recall that the ascent $p(T)$ and the descent of $T \in \mathcal{B}(X)$ are
$$p(T)=\inf\{ n \in \mathbb{N}: ker(T^{n})=ker(T^{n+1})  \}$$
and
$$q(T)=\inf\{ n \in \mathbb{N}: ran(T^{n})=ran(T^{n+1})  \},$$
respectively.  We set $p_{l}(x)=p(L_{x})$, $q_{l}(x)=q(L_{x})$, $p_{r}(x)=p(R_{x})$ and $q_{r}(x)=q(R_{x})$.

\begin {lemma} ${\label{3.2}}$  \begin{upshape}
  Let $x \in B\Phi(\mathcal{A})$.

  $(1)$ If $p_{l}(x)<\infty$, then there exists $\varepsilon >0$ such that $p_{l}(x-\lambda)=0$ for $0<|\lambda|<\varepsilon$.

   $(2)$ If $q_{l}(x)<\infty$, then there exists $\varepsilon >0$ such that $q_{l}(x-\lambda)=0$ for $0<|\lambda|<\varepsilon$.
 \end{upshape}
\end {lemma}

\begin{proof} (1) By Theorem \ref{2.6}(2), there exists $\varepsilon >0$ such that $null(x-\lambda)=\Theta(R(x)\cap x^{m}\mathcal{A})$
for sufficiently large $m \in \mathbb{N}$ and $0<|\lambda|<\varepsilon$.
Since $p_{l}(x)<\infty$, we get $R(x^{m})=R(x^{m+1})$ when $m \geq p_{l}(x)$.
Because $\frac{R(x^{m+1})}{R(x^{m})} \simeq R(x)\cap x^{m}\mathcal{A}$, we derive that $null(x-\lambda)=0$, which is equivalent to $p_{l}(x-\lambda)=0$.

(2) Theorem \ref{2.6}(3) ensures that there is $\varepsilon >0$ such that for sufficiently large $m \in \mathbb{N}$ and $0<|\lambda|<\varepsilon$,  $def(x-\lambda)=\Theta(L(x)\cap \mathcal{A}x^{m})$. In additional, as $q_{l}(x)<\infty$, we can find $m \geq q_{l}(x)$ such that $x^{m}\mathcal{A}=x^{m+1}\mathcal{A}$. If $a \in L(x^{m+1})$, then $ax^{m}=ax^{m+1}c=0$ for some $c \in \mathcal{A}$, thus $a \in L(x^{m})$. Therefore, $L(x^{m+1})=L(x^{m})$. From the fact $L(x)\cap \mathcal{A}x^{m} \simeq \frac{L(x^{m+1})}{L(x^{m})}$, it follows that $def(x-\lambda)=0$, i.e., $L(x-\lambda) =\{ 0\}$. Now $x-\lambda$ is Fredholm, hence $(x-\lambda)y_{\lambda}(x-\lambda)=x-\lambda$ for some $y_{\lambda} \in \mathcal{A}$. Let $p_{\lambda}=(x-\lambda)y_{\lambda}$. Then we have $\mathcal{A}(1-p_{\lambda})=L(x-\lambda)=\{0\}$, and therefore $p_{\lambda}=1$.
Hence $\mathcal{A}=p_{\lambda}\mathcal{A}\subseteq (x-\lambda)\mathcal{A}\subseteq \mathcal{A}.$ Consequently, $(x-\lambda)\mathcal{A}=\mathcal{A}$, which is equivalent to $q_{l}(x-\lambda)=0$.
\end{proof}

By the classical Baire category theorem, it follows that the finiteness of $p_{l}(x)$ (resp. $p_{r}(x)$) is equivalent to the closeness of $\bigcup \limits_{n=1}^{\infty} R(x^{n})$ (resp. $\bigcup \limits_{n=1}^{\infty} L(x^{n})$). The next result shows that the finiteness of $p_{l}(x)$ for B-Fredholm elements may be also characterized by various ways including in particular the left SVEP at $0$, the closeness of the left quasinilpotent part $H_{l}(x)$, and the accumulation points of the left spectrum $\sigma_{l}(x)$.

\begin {theorem} ${\label{3.3}}$  \begin{upshape}
  Let $\lambda_{0} \in \mathbb{C}$ and $x-\lambda_{0} \in B\Phi(\mathcal{A})$. Then the following assertions are equivalent:

  $(1)$ $x$ has left SVEP at $\lambda_{0}$;

  $(2)$ $p_{l}(x-\lambda_{0}) < \infty$;

  $(3)$ $q_{r}(x-\lambda_{0}) < \infty$;

  $(4)$ $\sigma_{l}(x)$ does not cluster at $\lambda_{0}$;

  $(5)$ $\lambda_{0}$ is not an interior point of $\sigma_{l}(x)$;

  $(6)$ $H_{l}(x-\lambda_{0})=R((x-\lambda_{0})^{p})$ for some $p \in \mathbb{N}$;

   $(7)$ $H_{l}(x-\lambda_{0})$ is closed;

   $(8)$ $H_{l}(x-\lambda_{0}) \cap K_{l}(x-\lambda_{0})=\{0\}$;

   $(9)$ $H_{l}(x-\lambda_{0}) \cap K_{l}(x-\lambda_{0})$ is closed;

   $(10)$ $\bigcup \limits_{n=1}^{\infty} R((x-\lambda_{0})^{n}) \cap \bigcap \limits_{n=1}^{\infty}(x-\lambda_{0})^{n} \mathcal{A} =\{0\}$.

   In this case, if $p:=p_{l}(x-\lambda_{0})$, then
   $$H_{l}(x-\lambda_{0})=\bigcup \limits_{n=1}^{\infty} R((x-\lambda_{0})^{n})=R((x-\lambda_{0})^{p}).$$
 \end{upshape}
\end {theorem}

\begin{proof} Without loss of generality, we assume that $\lambda_{0}=0$.

$(1) \Longrightarrow (2)$ Suppose that $p_{l}(x) = \infty$. The B-Fredholmness of $x$ implies that
$$\makebox{ the left hyperrange } M:=\bigcap\limits_{n=1}^{\infty}x^{n}\mathcal{A} \makebox{ is closed, }xM=M$$
and there exists a sufficiently large $m \in \mathbb{N}$ such that
$$R(x)\cap x^{m}\mathcal{A} = R(x)\cap  M.$$
Now the infiniteness of $p_{l}(x)$ implies that there is a nonzero $a \in R(x)\cap  M$.
By the open mapping theorem, we can find a constant $\alpha>0$ and a sequence $\{a_{n}\}_{n=1}^{\infty}$ in $M$ such that
$xa_{1}=a, xa_{n+1}=a_{n}, \makebox{ and } ||a_{n}||\leq \alpha^{n} ||a||.$ Let $U=\{u \in \mathbb{C}: |u| < \frac{1}{\alpha}\}$ and we define $f:U \longrightarrow \mathcal{A}$ by $f(u)=a+\sum\limits_{n=1}^{\infty}u^{n}a_{n}$ for $u \in U.$ Clearly, $f$ is a holomorphic function on $U$ and
$(u-x)f(u)=-xa=0,$ but $f \not\equiv 0$. This contradicts our assumption that $x$ has left SVEP at $0$.

$(2) \Longrightarrow (4)$ Since $p_{l}(x) < \infty$, by Lemma \ref{3.2}(1), there exists $\varepsilon >0$ such that $p_{l}(x-\lambda)=0$ for $0<|\lambda|<\varepsilon$. But $x-\lambda$ is Fredholm, so $x-\lambda$ relatively regular, and thus $x-\lambda$ is left invertible.
Therefore, $0$ is not a limit point of $\sigma_{l}(x)$.

$(4) \Longrightarrow (5)$ It is obvious.

$(5) \Longrightarrow (1)$  It is an immediate consequence of the identity theorem for analytic functions.

$(2) \Longleftrightarrow (3)$ Suppose first that $n=q_{r}(x)<\infty$. Then $\mathcal{A}x^{n}=\mathcal{A}x^{n+1}$, so $x^{n}=ax^{n+1}$ for some
$a \in \mathcal{A}$. For $b \in R(x^{n+1})$, we have $x^{n}b=ax^{n+1}b=0$, thus $b \in R(x^{n})$. This shows that $R(x^{n+1}) \subseteq R(x^{n})$,
therefore $p_{l}(x) \leq n$.

Conversely, suppose that $p_{l}(x) < \infty$. The B-Fredholmness of $x$ implies that $x^{m}$ and $x^{2m}$ are relatively regular for a sufficiently large integer $m \geq p_{l}(x)$. Now we have $R(x^{m})=R(x^{2m})$, $\mathcal{A}x^{m}=\mathcal{A}p$ and $\mathcal{A}x^{2m}=\mathcal{A}q$ for some idempotents $p,q \in \mathcal{A}$. Hence $(1-p)\mathcal{A}=R(x^{m})=R(x^{2m})=(1-q)\mathcal{A}$, so $(1-q)=(1-p)(1-q)$, and thus $p=pq$.
Consequently, $\mathcal{A}x^{m}=\mathcal{A}p=\mathcal{A}pq \subseteq \mathcal{A}q=\mathcal{A}x^{2m}$. This shows that $q_{r}(x)\leq m <\infty$.

$(2) \Longrightarrow (6)$  The B-Fredholmness of $x$ implies that $x^{m}\mathcal{A}$ is closed for a sufficiently large $m \in \mathbb{N}$.
As $p_{l}(x) < \infty$, by [\cite{Mbekhta-Muller-axiomatic}, Lemma 7] we know that $x^{n}\mathcal{A}$ is closed for all $n \geq p_{l}(x)$.
Hence by [\cite{Bel-Burgos-Oudghiri}, Proposition 4.1], $\overline{H_{l}(x)}=\overline{\bigcup \limits_{n=1}^{\infty} R(x^{n})}$.
Let $p=p_{l}(x)$. Then $H_{l}(x) \subseteq \overline{H_{l}(x)}=\overline{\bigcup \limits_{n=1}^{\infty} R(x^{n})}=R(x^{p}) \subseteq H_{l}(x)$.
Therefore, $H_{l}(x)=R(x^{p})$.

$(6) \Longrightarrow (7)$ It is obvious.

$(7) \Longrightarrow (8)$ and $(8) \Longleftrightarrow (9)$ It follows from [\cite{Aiena-book}, Theorem 2.31] by considering the left multiplication operator $L_{x}$.

$(8) \Longrightarrow (10)$ Clearly, $\bigcap\limits_{n=1}^{\infty}x^{n}\mathcal{A} \subseteq K_{l}(x)$. Since $M:=\bigcap\limits_{n=1}^{\infty}x^{n}\mathcal{A}$ is closed and $xM=M$, we get $\bigcap\limits_{n=1}^{\infty}x^{n}\mathcal{A} \subseteq K_{l}(x)$ by the open mapping theorem. Hence $\bigcap\limits_{n=1}^{\infty}x^{n}\mathcal{A} = K_{l}(x)$. Consequently,
$\bigcup \limits_{n=1}^{\infty} R(x^{n}) \cap \bigcap \limits_{n=1}^{\infty}x^{n} \mathcal{A}\subseteq H_{l}(x) \cap K_{l}(x) =\{0\}$.

$(10) \Longrightarrow (1)$ It follows from [\cite{Aiena-book}, Corollary 2.26] by considering the left multiplication operator $L_{x}$.
\end{proof}

Dually, the right SVEP at $0$ for B-Fredholm elements can be characterized by various ways including in particular, the finiteness of $q_{l}(x)$, the closeness of the right quasinilpotent part $H_{r}(x)$, and the accumulation points of the right spectrum $\sigma_{r}(x)$.

\begin {theorem} ${\label{3.4}}$  \begin{upshape}
  Let $\lambda_{0} \in \mathbb{C}$ and $x-\lambda_{0} \in B\Phi(\mathcal{A})$. Then the following assertions are equivalent:

  $(1)$ $x$ has right SVEP at $\lambda_{0}$;

  $(2)$ $p_{r}(x-\lambda_{0}) < \infty$;

  $(3)$ $q_{l}(x-\lambda_{0}) < \infty$;

  $(4)$ $\sigma_{r}(x)$ does not cluster at $\lambda_{0}$;

  $(5)$ $\lambda_{0}$ is not an interior point of $\sigma_{r}(x)$;

  $(6)$ $H_{r}(x-\lambda_{0})=L((x-\lambda_{0})^{p})$ for some $p \in \mathbb{N}$;

   $(7)$ $H_{r}(x-\lambda_{0})$ is closed;

   $(8)$ $H_{r}(x-\lambda_{0}) \cap K_{r}(x-\lambda_{0})=\{0\}$;

   $(9)$ $H_{r}(x-\lambda_{0}) \cap K_{r}(x-\lambda_{0})$ is closed;

   $(10)$ $\bigcup \limits_{n=1}^{\infty} L((x-\lambda_{0})^{n}) \cap \bigcap \limits_{n=1}^{\infty}\mathcal{A}(x-\lambda_{0})^{n} =\{0\}$.

   In this case, if $p:=p_{r}(x-\lambda_{0})$, then
   $$H_{r}(x-\lambda_{0})=\bigcup \limits_{n=1}^{\infty} L((x-\lambda_{0})^{n})=L((x-\lambda_{0})^{p}).$$
 \end{upshape}
\end {theorem}

\begin{proof}
The proof is similar to that of Theorem \ref{3.3}, we omit it here.
\end{proof}

\section{Classification of components of B-Fredholm resolvent set} \label{sec4}

Recall that an element $a \in \mathcal{A}$ is called a left (resp. right) topological divisor of zero if there exists a sequence $\{a_{n}\}_{n=1}^{\infty}$ in $\mathcal{A}$ such that $||a_{n}||=1$ for all $n$ and $aa_{n}\rightarrow 0$ (resp. $a_{n}a\rightarrow 0$). An element which is either a left or right topological divisor of zero is called a topological divisor of zero. If there exists a sequence $\{a_{n}\}_{n=1}^{\infty}$ in $\mathcal{A}$, each  $a_{n}$ of norm one, such that $aa_{n}\rightarrow 0$ and $a_{n}a\rightarrow 0$, then we call $a \in \mathcal{A}$ is a two-sided topological divisor of zero. It is clear that if $a$ is left (resp. right) invertible then $a$ is not a left (resp. right) topological divisor of zero.

For $x \in \mathcal{A}$, the B-Fredholm spectrum $\sigma_{BF}(x)$ of $x$ is defined as those complex numbers $\lambda$ for which $x-\lambda$ is not B-Fredholm.
The B-Fredholm resolvent set of $x$ is then defined as $\rho_{BF}(x)=\mathbb{C} \backslash \sigma_{BF}(x)$. From the characterization of the left SVEP at a point for B-Fredholm elements established in Theorem \ref{3.3}, we now obtain the following classification of components of $\rho_{BF}(x)$.

\begin {theorem} ${\label{4.1}}$  \begin{upshape}
  Let $x \in \mathcal{A}$ and $\Omega$ a component of $\rho_{BF}(x)$. Then the following alternative holds:

  $(1)$ $x$ has the left SVEP for every point of $\Omega$. In this case, $p_{l}(x-\lambda)<\infty$ for all $\lambda \in \Omega$.
   Moreover, $\sigma_{l}(x)$ does not have limit points in $\Omega$; $x-\lambda$ is not a left topological divisor of zero for every point $\lambda$ in $\Omega$, except at most countably many isolated points in $\Omega$.

  $(2)$ $x$ has the left SVEP at no point of $\Omega$. In this case, $p_{l}(x-\lambda)=\infty$ for all $\lambda \in \Omega$.
  $x-\lambda$ is a left topological divisor of zero for every point $\lambda$ in $\Omega$.
 \end{upshape}
\end {theorem}

\begin{proof} Let $S_{l}(x)=\{\lambda \in \Omega: x \makebox{ does not have the left SVEP at } \lambda \}$. The identity theorem for analytic functions implies that $S_{l}(x)$ is open. Next we show that $\Omega \backslash S_{l}(x)$ is also open. For this, let $ \lambda \in \Omega \backslash S_{l}(x)$. Then $p_{l}(x-\lambda)<\infty$ by Theorem \ref{3.3}.
Hence by Lemma \ref{3.2}(1) and the openness of $\Omega$, there exists $\varepsilon >0$ such that for all $0< |\mu-\lambda| < \varepsilon$, $p_{l}(x-\mu)=0<\infty$ and $\mu \in \Omega$. Therefore, again by Theorem \ref{3.3}, $x$ has the left SVEP at $\mu$.
This shows that $\mu \in \Omega \backslash S_{l}(x)$ for $|\mu-\lambda| < \varepsilon$. Because $\Omega$ is connected, $S_{l}(x)$ is empty or $S_{l}(x)=\Omega$.
That is, the alternative is established.

In case (1), by Theorem \ref{3.3}, $p_{l}(x-\lambda)<\infty$ for all $\lambda \in \Omega$ and $\sigma_{l}(x)$ does not have limit points in $\Omega$.
Consequently, $x-\lambda$ is left invertible, and thus $x-\lambda$ is not a left topological divisor of zero for every point $\lambda$ in $\Omega$,
except at most countably many isolated points in $\Omega$.

In case (2), again by Theorem \ref{3.3}, $p_{l}(x-\lambda)=\infty$ for all $\lambda \in \Omega$. Therefore, $R(x-\lambda) \neq \{0\}$, so
$x-\lambda$ is a left topological divisor of zero for every point $\lambda$ in $\Omega$.
\end{proof}

The proof of the following result is similar to that above, we omit it here.

\begin {theorem} ${\label{4.2}}$  \begin{upshape}
  Let $x \in \mathcal{A}$ and $\Omega$ a component of $\rho_{BF}(x)$. Then the following alternative holds:

  $(1)$ $x$ has the right SVEP for every point of $\Omega$. In this case, $q_{l}(x-\lambda)<\infty$ for all $\lambda \in \Omega$.
   Moreover, $\sigma_{r}(x)$ does not have limit points in $\Omega$; $x-\lambda$ is not a right topological divisor of zero for every point $\lambda$ in $\Omega$, except at most countably many isolated points in $\Omega$.

  $(2)$ $x$ has the right SVEP at no point of $\Omega$. In this case, $q_{l}(x-\lambda)=\infty$ for all $\lambda \in \Omega$.
  $x-\lambda$ is a right topological divisor of zero for every point $\lambda$ in $\Omega$.
 \end{upshape}
\end {theorem}

Combing Theorem \ref{4.1} with Theorem \ref{4.2}, we can get a further classification of the components of $\rho_{BF}(x)$.

\begin {theorem} ${\label{4.3}}$  \begin{upshape}
  Let $x \in \mathcal{A}$ and $\Omega$ a component of $\rho_{BF}(x)$. There are exactly the following four possibilities:

  $(1)$ $x$ has both the left SVEP and the right SVEP at every point of $\Omega$. In this case,
  $p_{l}(x-\lambda)=q_{l}(x-\lambda)<\infty$ for all $\lambda \in \Omega$.
   $\sigma(x)$ does not have limit points in $\Omega$. This case occurs exactly when $\Omega$ intersects the resolvent $\rho(x)$.

    $(2)$ $x$ has the left SVEP at every point of $\Omega$, whist $x$ fails to have the right SVEP for each point of $\Omega$.
    In this case, $p_{l}(x-\lambda)<\infty$ and $q_{l}(x-\lambda)=\infty$ for all $\lambda \in \Omega$.
    $\sigma_{l}(x)$ does not have limit points in $\Omega$ and $\Omega \subseteq \sigma_{r}(x)$

    $(3)$ $x$ has the right SVEP at every point of $\Omega$, whist $x$ fails to have the left SVEP for each point of $\Omega$.
    In this case, $p_{l}(x-\lambda)=\infty$ and $q_{l}(x-\lambda)<\infty$ for all $\lambda \in \Omega$.
   $\sigma_{r}(x)$ does not have limit points in $\Omega$ and $\Omega \subseteq \sigma_{l}(x)$.

  $(4)$ $x$ has neither the left SVEP nor the right SVEP at the points of $\Omega$. In this case, $p_{l}(x-\lambda)=q_{l}(x-\lambda)=\infty$ for all $\lambda \in \Omega$. $\Omega \subseteq \sigma_{l}(x) \cap \sigma_{r}(x)$.
 \end{upshape}
\end {theorem}

We conclude this section with some interesting applications of the classification of the components of $\rho_{BF}(x)$. Let $\Pi(x)$ denote the poles of the resolvent of $x$.

\begin {corollary} ${\label{4.4}}$  \begin{upshape}
  Let $x \in \mathcal{A}$. Then
  $$\rho_{BF}(x) \cap \partial\sigma(x)=\Pi(x).$$
Moreover, the following assertions are equivalent:

 (i) $\sigma_{BF}(x) =\emptyset$;

 (ii) $\partial\sigma(x) \subseteq \rho_{BF}(x)$;

 (iii) $x$ is algebraic.
 \end{upshape}
\end {corollary}

\begin{proof}
By [\cite{Boasso}, Theorem 12], the poles of the resolvent of $x$ are exactly the isolated points $\lambda$ of the spectrum $\sigma(x)$
such that $x-\lambda$ is Drazin invertible. Hence
$$\Pi(x)  \subseteq \rho_{BF}(x) \cap \partial\sigma(x).$$
For the other inclusion, suppose that $ \lambda \in \rho_{BF}(x) \cap \partial\sigma(x),$ then $\lambda$ belongs to some component $\Omega$ of $\rho_{BF}(x)$, which intersects the resolvent $\rho(x)$, so case (1) of Theorem \ref{4.3} occurs. Therefore, $p_{l}(x-\lambda)=q_{l}(x-\lambda)<\infty$, which is equivalent
to say $L_{x-\lambda}$ is Drazin invertible. By [\cite{Boasso}, Theorem 4], $x-\lambda$ is Drazin invertible, so $\lambda$ is a pole of the resolvent of $x$.

$\makebox{(i)} \Longrightarrow \makebox{(ii)}$ It is obvious.

$\makebox{(ii)} \Longrightarrow \makebox{(iii)}$ As the arguments above, we infer that if $\partial\sigma(x) \subseteq \rho_{BF}(x)$ then $\partial\sigma(x) \subseteq \rho_{D}(x)$,
where $\rho_{D}(x)=\{ \lambda \in \mathbb{C}: x-\lambda \makebox{ is Drazin invertible }\}.$ Consequently, $x$ is algebraic by
[\cite{Boasso-algebraic element in BA}, Theorem 2.1].

$\makebox{(iii)} \Longrightarrow \makebox{(i)}$ Again by [\cite{Boasso-algebraic element in BA}, Theorem 2.1], $\sigma_{D}(x) =\emptyset$, where
$\sigma_{D}(x) =\mathbb{C} \backslash \rho_{D}(x)$. Hence $\sigma_{BF}(x) =\emptyset$, as we know that $\sigma_{BF}(x) \subseteq \sigma_{D}(x)$.
\end{proof}

\begin {corollary} ${\label{4.5}}$  \begin{upshape}
The following assertions are equivalent:

(i) $x$ is B-Fredholm for each $x \in \mathcal{A}$;

(ii) $\mathcal{A}$ is algebraic, that is all elements in $\mathcal{A}$ are algebraic.

\noindent Moreover, if $\mathcal{A}$ is semisimple, then (i) and (ii) are
equivalent to:

(iii) $\mathcal{A}$ is finite dimensional.
 \end{upshape}
\end {corollary}

\begin{proof}
$\makebox{(i)} \Longrightarrow \makebox{(ii)}$ For each $x \in \mathcal{A}$, since $x-\lambda$ is B-Fredholm for all $\lambda \in \mathbb{C}$, we know that
$\sigma_{BF}(x)=\emptyset$. By Corollary \ref{4.4}, $x$ is algebraic. Consequently, $\mathcal{A}$ is algebraic.

$\makebox{(ii)} \Longrightarrow \makebox{(i)}$ By Corollary \ref{4.4} again, $\sigma_{BF}(x)=\emptyset$ for each $x \in \mathcal{A}$, and thus
$x$ is B-Fredholm.

$\makebox{(iii)} \Longrightarrow \makebox{(i)}$ It is obvious.

$\makebox{(ii)} \Longrightarrow \makebox{(iii)}$ According to [\cite{Aupetit-book}, Theorem 5.4.2] we infer that if $\mathcal{A}$ is algebraic and semisimple,
then $\mathcal{A}$ is finite dimensional.

\end{proof}

\begin {corollary} ${\label{4.6}}$  \begin{upshape}
  Let $x \in \mathcal{A}$. Then we have
  $$\partial\sigma(x) \subseteq \sigma_{BF}(x) \cup \Pi(x).$$
 \end{upshape}
\end {corollary}

\begin {corollary} ${\label{4.7}}$  \begin{upshape}
  Let $x \in \mathcal{A}$ and $\Omega$ a component of $\rho_{BF}(x)$. Then we have
  $$\Omega \subseteq \sigma(x)\makebox{ or } \Omega \backslash \Pi(x) \subseteq \rho(x),$$
 \end{upshape}
\end {corollary}

\begin{proof}
In the cases (2), (3) and (4) of Theorem \ref{4.3}, we can see that $\Omega \subseteq \sigma(x)$.
In case (1) of Theorem \ref{4.3}, we have that $p_{l}(x-\lambda)=q_{l}(x-\lambda)<\infty$ for all $\lambda \in \Omega$.
Hence, for $\lambda \in \Omega \backslash \Pi(x)$, $x-\lambda$ is invertible by [\cite{Boasso}, Theorem 12].
Consequently, $\Omega \backslash \Pi(x) \subseteq \rho(x)$.
\end{proof}

\begin {corollary} ${\label{4.8}}$  \begin{upshape}
  Let $x \in \mathcal{A}$. Then we have
  $$\sigma(x) \makebox{ is at most countable} \Longleftrightarrow \sigma_{BF}(x) \makebox{ is at most countable}.$$
  In this case, $\sigma(x)=\sigma_{BF}(x) \cup \Pi(x).$
 \end{upshape}
\end {corollary}

\begin{proof}
Suppose that $\sigma_{BF}(x)$ is at most countable, then $\rho_{BF}(x)$ is the only connected component which intersects the resolvent $\rho(x)$.
According to Corollary \ref{4.7}, $\rho_{BF}(x) \backslash \Pi(x) \subseteq \rho(x)$.
Consequently, $$\sigma(x)=\sigma_{BF}(x)\cup \Pi(x)$$
is countable, which completes the proof.
\end{proof}

An element $x \in \mathcal{A}$ is called meromorphic if every non-zero points of its spectrum are poles of the resolvent of $x$.
Note that if $\sigma_{BF}(x) \subseteq \{0\}$, then $\rho_{BF}(x)$ has only one component. As a result, the following corollary
is also a direct consequence of  Theorem \ref{4.3}.

\begin {corollary} ${\label{4.9}}$  \begin{upshape}
  Let $x \in \mathcal{A}$. Then we have
  $$x \makebox{ is meromorphic} \Longleftrightarrow  \sigma_{BF}(x) \subseteq \{0\}.$$
 \end{upshape}
\end {corollary}

\section{B-Fredholm spectrum and perturbations} \label{sec5}

The main concern in the subsequent two sections is the intrinsic characterizations, from two different perspectives, of the following class of elements in $\mathcal{A}$,
$$\mathcal{F}:=\{f \in \mathcal{A}: f^{n} \in \mathrm{soc}(\mathcal{A}) \makebox{ for some }n \in \mathbb{N}\}.$$
In this section, we characterize elements in $\mathcal{F}$ by perturbation theory. Precisely, it is shown that
the B-Fredholm spectrum is invariant under any commuting perturbation $f \in \mathcal{F}$, and conversely this perturbation property characterizes such elements $f$ in the case that $\mathcal{A}$ is semisimple. This investigation dates back to an earlier result of M.A. Kaashoek and D.C. Lay in 1972, see [\cite{Kaashoek-Lay}, Theorem 2.2]. When $\mathcal{A}=\mathcal{B}(X)$, they showed that the descent spectrum is invariant under any commuting perturbation $F$ such that $F^{n}$ is of finite rank for some $n \in \mathbb{N}$. They also conjectured that this perturbation property characterizes such operators $F$. In 2006, Burgos, Kaidi, Mbekhta and Oudghiri [\cite{Burgos-Kaidi-Mbekhta-Oudghiri}, Theorem 3.1] provided an affirmative answer to this conjecture.
Later, this result is generalized to various spectra. In particular, Zeng, Jiang and Zhong extended this result to B-Fredholm spectrum [\cite{Zeng-Jiang-Zhong}, Theorem 2.1] by using the theory of operators with eventual topological uniform descent, see [\cite{Zeng-Jiang-Zhong}] for details.

Ha\"{\i}ly, Kaidi and Rodr\'{\i}guez Palacios extended [\cite{Burgos-Kaidi-Mbekhta-Oudghiri}, Theorem 3.1] to the descent spectrum in semisimple Banach algebras, see [\cite{Haily-Kaidi-Palacios}, Theorem 3.6]. By using the characterization of algebraic elements (see Corollary \ref{4.4}) and some techniques developed in [\cite{Haily-Kaidi-Palacios}], we shall prove a variant of [\cite{Zeng-Jiang-Zhong}, Theorem 2.1] for B-Fredholm spectrum in semisimple Banach algebras.

To do this we first need a preliminary result concerning Drazin invertibility.

\begin {lemma} ${\label{5.1}}$  \begin{upshape} Let $\mathcal{A}$ be an algebra with a unit. If $a \in \mathcal{A}$ is Drazin invertible and $b$ is a nilpotent element commuting with $a$, then $a+b$ is also Drazin invertible.
 \end{upshape}
\end {lemma}

\begin{proof} Since $a \in \mathcal{A}$ is Drazin invertible,  by [\cite{Berkani-rings}, Proposition 2.5]
we infer that the left multiplication operator $L_{a}$ has finite ascent and descent. Note that $L_{b}$ is a nilpotent linear operator which commutes with $L_{a}$. Therefore, according to a classical result of Kaashoek and Lay ([\cite{Kaashoek-Lay}, Theorem 2.2]),
$L_{a+b}$ also have finite ascent and descent. This is equivalent to say $a+b$ is Drazin invertible by [\cite{Aiena-book},  Theorem 3.6] and [\cite{Berkani-rings}, Proposition 2.5].
\end{proof}

Following Aupetit and Mouton [\cite{Aupetit-Mouton}], a trace function on the socle is defined by $\tau(a)=\Sigma_{\lambda \in \sigma(a)} \lambda m(\lambda,a)$ for $a \in soc(\mathcal{A})$, where $m(\lambda,a)$ is the algebraic multiplicity of $\lambda$ for $a$.
With the aid of the trace function, the index for B-Fredholm elements was introduced in [\cite{Berkani-BFredholm}, Definition 2.2].

\begin {definition} ${\label{5.2}}$  \begin{upshape}
The index of a B-Fredholm element $a \in \mathcal{A}$ is defined by
$$\makebox{i}(a)=\tau(aa_{0}-a_{0}a),$$
where $\pi(a_{0})$ is a Drazin inverse of $\pi(a)$.
 \end{upshape}
\end {definition}

According to [\cite{Berkani-BFredholm}, Theorem 2.3], the index of a B-Fredholm element $a \in \mathcal{A}$ is well defined and is independent of $a_{0}$.

It is well known (see [\cite{Barnes-Murphy-Smyth-West}, Theorem F.1.10]) that $a \in \mathcal{A}$ is Fredholm if and only if $\mathcal{A}a=\mathcal{A}(1-q)$ and $a\mathcal{A}=(1-p)\mathcal{A}$ for some idempotents $p,q \in soc(\mathcal{A})$. In this case, we say that $q$ is a right Barnes idempotent for $a$, and $p$ is a left Barnes idempotent for $a$. The Fredholm index of a Fredholm element $a \in \mathcal{A}$ is given by $\makebox{i}(a)=null(a)-def(a)$,
see [\cite{Barnes-Fredholm-element}, Definition 3.1]. According to [\cite{Grobler-Raubenheimer}, Theorems 3.14 and 3.17], the Fredholm index and the B-Fredholm index coincide for Fredholm elements.

Recall that an algebra $\mathcal{A}$ is said to be semisimple if its Jacobson radical $rad(\mathcal{A})$ is equal to $\{0\} $. We say that $\mathcal{A}$ is primitive if it possesses a faithful irreducible representation. It is well known that
$$ \makebox{primitive} \Longrightarrow \makebox{semisimple} \Longrightarrow \makebox{semiprime}.$$

\begin {theorem} ${\label{5.3}}$  \begin{upshape} Let $f \in \mathcal{A}$ with $f^{n} \in \mathrm{soc}(\mathcal{A})$ for some $n \in \mathbb{N}$. If $x \in B\Phi(\mathcal{A})$ commutes with $f$, then $x + f \in B\Phi(\mathcal{A})$. If, additionally, $\mathcal{A}$ is primitive then
$$\makebox{i}(x+f)=\makebox{i}(x).$$
 \end{upshape}
\end {theorem}

\begin{proof} Since $x$ is B-redholm, $\pi(x)$ is Drazin invertible. Observe that $\pi(f)$ is a nilpotent element commuting with $\pi(x)$. It follows from Lemma \ref{5.1} that $\pi(x+f)=\pi(x)+\pi(f)$ is also Drazin invertible. That is, $x + f$ is B-Fredholm.

For the index equality, we consider the canonical map $\phi:\mathcal{A} \longrightarrow \mathcal{A} /\overline{ soc(\mathcal{A})}$. By [\cite{Berkani-BFredholm}, Theorem 3.1], there exists  $\varepsilon >0$ such that for $0<|\lambda|<\varepsilon$, $x-\lambda$ is Fredholm, or equivalently,
$\phi(x-\lambda)$ is invertible in the Banach algebra $\mathcal{A} /\overline{ soc(\mathcal{A})}$. For $\mu \in [0,1]$, it is clear that $\phi(\mu f)$ is a nilpotent element commuting with $\phi(x-\lambda)$. We claim that $\phi(x-\lambda+\mu f)$ is invertible. Our claim follows from the following fact:

If $a$ is an invertible element in a unital algebra, $b$ is a nilpotent element commuting with $a$, then $a+b$ is also invertible.

Indeed, $(1+a^{-1}b)(1-a^{-1}b+a^{-2}b^{2}-\cdots+(-1)^{n-1}a^{-(n-1)}b^{n-1})=1+(-1)^{n-1}a^{-n}b^{n}=1.$
Hence $a+b=a(1+a^{-1}b)$ is invertible.

Now the path $\{ x-\lambda+\mu f: \mu \in [0,1]\}$ lies in the set of Fredholm elements in $\mathcal{A}$. By the stability of the Fredholm index (see [\cite{Barnes-Fredholm-element}, Theorem 4.1]), it follows that $\makebox{i}(x-\lambda)=\makebox{i}(x-\lambda+f)$.
Again by [\cite{Berkani-BFredholm}, Theorem 3.1], $\makebox{i}(x)= \makebox{i}(x-\lambda)$ and $\makebox{i}(x+f)= \makebox{i}(x+f-\lambda)$ for sufficiently small $\lambda$. Consequently, $\makebox{i}(x+f)=\makebox{i}(x).$
\end{proof}

An element $a \in \mathcal{A}$ is called B-Weyl if it is B-Fredholm of index zero. The B-Weyl spectrum of $a$ is then defined by
$$\sigma_{BW}(a)=\{\lambda \in \mathbb{C}: a-\lambda \makebox{ is not B-Weyl} \}.$$
Clearly, $\sigma_{BF}(a) \subseteq \sigma_{BW}(a) \subseteq \sigma_{D}(a)$. Now Combing Corollary \ref{4.4} and
[\cite{Boasso-algebraic element in BA}, Theorem 2.1], we infer that
\begin{equation} {\label{eq5.4}}
 \qquad\qquad\qquad\qquad\qquad  a \makebox{ is algebraic} \Longleftrightarrow \sigma_{BW}(a)=\emptyset.
 \end{equation}

Let $p \in \mathcal{A}$ be an idempotent. Clearly $p\mathcal{A}p$ is a closed subalgebra of $\mathcal{A}$ with identity $p$.
For $b \in p\mathcal{A}p$, in order to avoid confusion, we let
$$\sigma_{BF}(b, \mathcal{A})=\{\lambda \in \mathbb{C}: b-\lambda \makebox{ is not B-Fredholm in } \mathcal{A}\}$$
and
$$\sigma_{BF}(b, p\mathcal{A}p)=\{\lambda \in \mathbb{C}: b-\lambda p \makebox{ is not B-Fredholm in } p\mathcal{A}p\}.$$
When no ambiguity is possible, we write $\sigma_{BF}(a)$ instead of $\sigma_{BF}(a, \mathcal{A})$ for $a \in \mathcal{A}$ as before.
For other spectra, we adopt analogous notations.

\begin {lemma} ${\label{5.4}}$ \begin{upshape} Let $\mathcal{A}$ be a unital semisimple Banach algebra. If $p \in \mathcal{A}$ is an idempotent commuting with $a \in \mathcal{A}$, then
\begin{equation} {\label{eq5.1}}
\qquad\qquad\qquad\qquad \ \   \sigma_{BF}(ap, p\mathcal{A}p)= \sigma_{BF}(ap, \mathcal{A})
\end{equation}
\begin{equation} {\label{eq5.2}}
\qquad\qquad\qquad\qquad \ \   \sigma_{BW}(ap, p\mathcal{A}p)= \sigma_{BW}(ap, \mathcal{A})
\end{equation}
 and
 \begin{equation} {\label{eq5.3}}
 \qquad\qquad \ \  \sigma_{BF}(a, \mathcal{A})=\sigma_{BF}(ap, \mathcal{A}) \cup \sigma_{BF}(a(1-p), \mathcal{A}).
 \end{equation}
 \end{upshape}
\end {lemma}

\begin{proof} We use the fact $soc(p\mathcal{A}p)=psoc(\mathcal{A})p$ as observed by Barnes in [\cite{Barnes-algebraic-element}, p. 229]. This fact is crucial in the following proof.

Suppose that $\lambda \notin \sigma_{BF}(ap, \mathcal{A})$, i.e., $\pi(ap-\lambda)$ is Drazin invertible in $\mathcal{A} / soc(\mathcal{A})$. Then there is $b \in \mathcal{A}$ such that $\pi(ap-\lambda)\pi(b)=\pi(b)\pi(ap-\lambda), \pi(b)\pi(ap-\lambda)\pi(b)=\pi(b) \makebox{ and }$  $$\pi(ap-\lambda)\pi(b)\pi(ap-\lambda)-\pi(ap-\lambda) \makebox{ is nilpotent}.$$
Let $\phi:p\mathcal{A}p\rightarrow p\mathcal{A}p/ p\mathrm{soc}(\mathcal{A})p$ be the canonical map. A direct computation shows that $\phi(pbp)$ is the Drazin inverse of $\phi(ap-\lambda p)$ in the quotient algebra $p\mathcal{A}p/ p\mathrm{soc}(\mathcal{A})p$. This shows that $\lambda \notin \sigma_{BF}(ap, p\mathcal{A}p)$. Conversely, suppose that $\lambda \notin \sigma_{BF}(ap, p\mathcal{A}p)$. Then there exist $b=pbp \in p\mathcal{A}p$ and $k \in \mathbb{N}$ such that
$\phi(ap-\lambda p)\phi(b)=\phi(b)\phi(ap-\lambda p),$
$$\phi(b)\phi(ap-\lambda p)\phi(b)=\phi(b) \makebox{ and } \phi^{k}(ap-\lambda p)\phi(b)\phi(ap-\lambda p)=\phi^{k}(ap-\lambda p).$$
Clearly, if $\lambda=0$ then $\pi(b)$ is the Drazin inverse of $\pi(ap)$ in the quotient algebra $\mathcal{A}/\mathrm{soc}(\mathcal{A})$.
Consider the other case $\lambda\neq0$. Let $q=1-p$. Note that $ap-\lambda=ap-\lambda p-\lambda q$. Then we have
$$\pi(ap-\lambda)\pi(b-\frac{1}{\lambda}q)=\pi(b-\frac{1}{\lambda}q)\pi(ap-\lambda),$$
$$\pi(b-\frac{1}{\lambda}q)\pi(ap-\lambda)\pi(b-\frac{1}{\lambda}q)=\pi(b-\frac{1}{\lambda}q)$$
and
$$\pi^{k}(ap-\lambda) \pi(b-\frac{1}{\lambda}q)\pi(ap-\lambda)=\pi^{k}(ap-\lambda).$$
Therefore, $\lambda \notin \sigma_{BF}(ap, \mathcal{A})$. This completes the proof of (\ref{eq5.1}).

To prove (\ref{eq5.2}), from the above arguments, it remains to show that if $\lambda \notin \sigma_{BF}(ap, p\mathcal{A}p)$, then
$$\makebox{i}(ap- \lambda p)=\makebox{i}(ap- \lambda).$$
When $\lambda =0$, there is nothing to prove. If $\lambda \neq 0$, by the definition of the B-Fredholm index,
$$\makebox{i}(ap- \lambda p)=\tau[(ap-\lambda p)b-b(ap-\lambda p)]=\tau (ab-ba)$$
and
$$\makebox{i}(ap- \lambda)=\tau[(ap-\lambda)(b-\frac{1}{\lambda}q)-(b-\frac{1}{\lambda}q)(ap-\lambda)]=\tau (ab-ba).$$

To prove (\ref{eq5.3}), by (\ref{eq5.1}), it remains to show that
$$\sigma_{BF}(a, \mathcal{A})=\sigma_{BF}(ap, p\mathcal{A}p) \cup \sigma_{BF}(aq, q\mathcal{A}q),$$
where $q=1-p$.
Suppose that $\lambda \notin \sigma_{BF}(a, \mathcal{A})$, then
$\pi(a-\lambda)$ has a Drazin inverse $\pi(b)$ in quotient algebra $\mathcal{A} / soc(\mathcal{A})$, for some $b \in \mathcal{A}$.
A simple computation shows that $\pi(pbp)$ (resp. $\pi(qbq)$) is the Drazin inverse of $\pi(ap-\lambda p)$ (resp. $\pi(aq-\lambda q)$) in
the quotient algebra $p\mathcal{A}p/ p\mathrm{soc}(\mathcal{A})p$ (resp. $q\mathcal{A}q/ q\mathrm{soc}(\mathcal{A})q$). This shows that
$\lambda \notin  \sigma_{BF}(ap, p\mathcal{A}p) \cup \sigma_{BF}(aq, q\mathcal{A}q)$.
Conversely, let $\lambda \notin  \sigma_{BF}(ap, p\mathcal{A}p) \cup \sigma_{BF}(aq, q\mathcal{A}q)$. Then there exists $b \in p\mathcal{A}p$ (resp. $c \in q\mathcal{A}q$) such that $\pi(ap-\lambda p)$ (resp. $\pi(aq-\lambda q)$) has a Drazin inverse $\pi(b)$ (resp. $\pi(c)$) in the quotient algebra $p\mathcal{A}p/ p\mathrm{soc}(\mathcal{A})p$ (resp. $q\mathcal{A}q/ q\mathrm{soc}(\mathcal{A})q$). Hence
$$\pi(a-\lambda)\pi(b+c)=\pi(b+c)\pi(a-\lambda),$$
$$\pi(b+c)\pi(a-\lambda)\pi(b+c)=\pi(b+c)$$
and
$$\pi^{k}(a-\lambda)\pi(b+c)\pi(a-\lambda)=\pi^{k}(a-\lambda),$$
when $k$ is grater than the Drazin index of $\pi(ap-\lambda p)$ and $\pi(aq-\lambda q)$.
This shows that $\lambda \notin \sigma_{BF}(a, \mathcal{A})$, and completes the proof of the equality (\ref{eq5.3}).
\end{proof}

We are now in a position to give the proof of the main result of this section.

\begin {theorem} ${\label{5.5}}$  \begin{upshape}
Let $\mathcal{A}$ be a unital semisimple Banach algebra and $f \in \mathcal{A}$. Then the following statements are equivalent:

(i) $f^{n} \in \mathrm{soc}(\mathcal{A})$ for some $n \in \mathbb{N}$;

(ii) $\sigma_{BF}(x+f) =\sigma_{BF}(x)$ for all $x \in \mathcal{A}$ commuting with $f$.


Additionally, if $\mathcal{A}$ is primitive then the above conditions are also equivalent to the following assertion:

(iii) $\sigma_{BW}(x+f) =\sigma_{BW}(x)$ for all $x \in \mathcal{A}$ commuting with $f$.
 \end{upshape}
\end {theorem}

\begin{proof}
$\makebox{(i)} \Longrightarrow \makebox{(ii)}$ By Theorem \ref{5.3}.

$\makebox{(ii)} \Longrightarrow \makebox{(i)}$ Taking $x=0$ in the assumption (ii), we get that $\sigma_{BF}(f)=\emptyset$.
Hence by Corollary \ref{4.4}, $f$ is algebraic.  Let
$p(\lambda)=(\lambda-\lambda_{1})^{k_{1}} (\lambda-\lambda_{2})^{k_{2}}\cdots (\lambda-\lambda_{n})^{k_{n}}$ be the minimal polynomial of $f$.
Observe that $p(L_{f})=0$. Now, by [\cite{Aiena-book}, Lemma 1.76], $\mathcal{A} = \bigoplus\limits_{i=1}^{n} R((f-\lambda_{i})^{k_{i}})$.
Hence there exist uniquely determined elements $p_{i} \in R((f-\lambda_{i})^{k_{i}})$ such that $1=\sum\limits_{i=1}^{n} p_{i}$. Clearly,
$p_{1},p_{2},\cdots,p_{n}$ are orthogonal idempotents commuting with $f$, such that $(f-\lambda_{i})^{k_{i}}p_{i}=0$ for every $i=1,2,\cdots,n$.
Precisely, $p_{i}$ is the spectral projection associated with $f$ and $\{\lambda_{i}\}$, for $1\leq i \leq n$.


We claim that $\dim p_{i} \mathcal{A} p_{i}<\infty$ when $\lambda_{i}\neq 0$. Suppose that this is not true. Then there exists
$\lambda := \lambda_{i}\neq 0$ and $ p:= p_{i}$ such that $\dim p \mathcal{A} p = \infty$. By [\cite{Haily-Kaidi-Palacios}, Theorem 2.3],
we may find a non-algebraic element $b$, which commutes with $fp$, in the semisimple Banach algebra $p \mathcal{A} p$. Clearly, $b$ commutes with $f$, and by
Corollary \ref{4.4},
\begin{equation} {\label{eq5.5}}
 \qquad\qquad\qquad\qquad\qquad\qquad   \sigma_{BF}(b,p\mathcal{A} p) \neq \emptyset.
 \end{equation}
By the assumption (ii),
$$\sigma_{BF}(b, \mathcal{A})=\sigma_{BF}(b+f, \mathcal{A}).$$
By Lemma \ref{5.4},
\begin{equation} {\label{eq5.6}}
 \qquad\qquad\qquad \sigma_{BF}(b+f, \mathcal{A}) = \sigma_{BF}(b+fp, \mathcal{A}) \cup \sigma_{BF}(f(1-p), \mathcal{A}).
 \end{equation}
Since $(f-\lambda)p$ is nilpotent, by Theorem \ref{5.3},
$$\sigma_{BF}(b+fp, \mathcal{A})= \sigma_{BF}(b+\lambda p, \mathcal{A}).$$
Again by Lemma \ref{5.4},
\begin{align*}
\qquad \qquad \lambda+\sigma_{BF}(b, p\mathcal{A}p)&=\sigma_{BF}(\lambda p+b, p\mathcal{A}p) \\
                                     &=\sigma_{BF}(\lambda p+b,\mathcal{A})\subseteq \sigma_{BF}(b, \mathcal{A})=\sigma_{BF}(b, p\mathcal{A}p).
\end{align*}
This contradicts to the facts that $\lambda \neq 0$ and $\sigma_{BF}(b, p\mathcal{A}p)$ is bounded.

Now by [\cite{Aiena-book}, Theorem 5.24], $p_{i} \in soc(\mathcal{A})$ when $\lambda_{i}\neq 0$. But $fp_{i}$ is nilpotent when $\lambda_{i}=0$.
Consequently, from $f=f(\sum\limits_{i=1}^{n} p_{i})$ we conclude that $f$ has the desired property.

$\makebox{(i)} \Longrightarrow \makebox{(iii)}$ By Theorem \ref{5.3} again.

$\makebox{(iii)} \Longrightarrow \makebox{(i)}$ Applying the proof of $\makebox{(ii)} \Longrightarrow \makebox{(i)}$ to the B-Weyl spectrum,
it only remains to replace (\ref{eq5.5}) with $\sigma_{BW}(b,p\mathcal{A} p) \neq \emptyset$, and replace (\ref{eq5.6}) with
\begin{equation} {\label{eq5.7}}
 \qquad\qquad\qquad\qquad\qquad  \sigma_{BW}(b+fp, \mathcal{A}) \subseteq \sigma_{BW}(b+f, \mathcal{A}).
 \end{equation}
But $\sigma_{BW}(b,p\mathcal{A} p) \neq \emptyset$ is a consequence of the equivalence (\ref{eq5.4}). Next we prove the inclusion (\ref{eq5.7}).

Note that, for $\lambda \in \mathbb{C}$,
$$(f-\lambda)(1-p)=(f-\lambda)\sum\limits_{j=1,j\neq i}^{n}p_{j}=\sum\limits_{j=1,j\neq i}^{n}(f-\lambda_{j})p_{j}+(\lambda_{j}-\lambda)p_{j}.$$
Since $(f-\lambda_{j})p_{j}$ is nilpotent and $(\lambda_{j}-\lambda)p_{j}$ is invertible or equals to zero in $p_{j}\mathcal{A}p_{j},$
$(f-\lambda_{j})p_{j}+(\lambda_{j}-\lambda)p_{j}$ is Drazin invertible in $p_{j}\mathcal{A}p_{j}.$ Thus, $(f-\lambda)(1-p)$ is Drazin invertible in $(1-p)\mathcal{A}(1-p).$ Therefore, $(b+f-\lambda)(1-p)$ is B-Weyl in $(1-p)\mathcal{A}(1-p).$

Let $\lambda \notin  \sigma_{BW}(b+f, \mathcal{A}).$ Then by Lemma \ref{5.4}, $(b+f-\lambda)p$ is B-Fredholm in $p\mathcal{A} p$. By the additivity of the trace (see [\cite{Aupetit-Mouton}, Theorem 3.3(i)), we infer that
$$\makebox{i}(b+f-\lambda) = \makebox{i}((b+f-\lambda)p) + \makebox{i}((b+f-\lambda)(1-p)).$$
This implies that $\makebox{i}((b+f-\lambda)p)=0$, and so $\lambda \notin \sigma_{BW}(b+fp, p\mathcal{A}p)$. By Lemma \ref{5.4} again, $\lambda \notin \sigma_{BW}(b+fp, \mathcal{A})$, which completes the proof of (\ref{eq5.7}).
\end{proof}

 \section{B-Fredholm elements which are Riesz} \label{sec6}

Throughout this section, we assume that $\mathcal{A}$ is a unital semisimple Banach algebra.
Following Pearlman [\cite{Pearlman-Riesz-points}], an element $a \in \mathcal{A}$ is called a Riesz element if $\phi(a)$ is quasinilpotent, where
$\phi:\mathcal{A} \longrightarrow \mathcal{A} /\overline{soc(\mathcal{A})}$ is the canonical quotient homomorphism. In the previous section, we characterize elements in the class
$$\mathcal{F}:=\{f \in \mathcal{A}: f^{n} \in \mathrm{soc}(\mathcal{A}) \makebox{ for some }n \in \mathbb{N}\},$$
by means of the commuting perturbational invariance of the B-Fredholm spectrum. In this section, we give some other characterizations of $\mathcal{F}$ from a different perspective. In particular, we show that $\mathcal{F}$ is precisely the intersection of the class of Riesz elements and the class of B-Fredholm elements. In order to do this we need the following characterization of Riesz elements, which is due to Pearlman.

\begin {lemma} ${\label{7.1}}$  \begin{upshape} ([\cite{Pearlman-Riesz-points}, Corollary 4.13]) Let $x \in \mathcal{A}$. Then $x$ is a Riesz element if and only if $x-\lambda$ is Fredholm and $p_{l}(x-\lambda)=q_{l}(x-\lambda)<\infty$ for all nonzero $\lambda \in \mathbb{C}$.
 \end{upshape}
\end {lemma}

\begin {theorem} ${\label{7.2}}$  \begin{upshape}
Let $\mathcal{A}$ be a semisimple Banach algebra and $f \in \mathcal{A}$. The following statements are equivalent:

(i) $f^{n} \in \mathrm{soc}(\mathcal{A})$ for some $n \in \mathbb{N}$;

(ii) $f$ is a Riesz and Drazin invertible element;

(iii) $f$ is a Riesz and B-Weyl element;

(iv) $f$ is a Riesz and B-Fredholm element.
 \end{upshape}
\end {theorem}

\begin{proof}
$\makebox{(i)} \Longrightarrow \makebox{(ii)}$ Since $f^{n} \in \mathrm{soc}(\mathcal{A})$, $f^{n}$ is Riesz, and hence $f$ is also Riesz.
Next we show that $f$ is Drazin invertible.

Noting that $f^{n} \in \mathrm{soc}(\mathcal{A})$, it follows that $\{f^{m}\mathcal{A}\}_{m=n}^{\infty}$ is a decreasing sequence of right ideals of finite order. Hence we can choose an integer $m \geq n$ such that $f^{m}\mathcal{A}=f^{m+1}\mathcal{A}$. This, together with Lemma \ref{7.1}, implies that $q_{l}(f-\lambda)<\infty$ for all $\lambda \in \mathbb{C}$. As a consequence of [\cite{Burgos-Kaidi-Mbekhta-Oudghiri}, Theorem 1.5], $f$ is algebraic. Hence by [\cite{Boasso-algebraic element in BA}, Theorem 2.1], $f$ is Drazin invertible.

$\makebox{(ii)} \Longrightarrow \makebox{(iii)} \Longrightarrow \makebox{(iv)}$ Clear.

$\makebox{(iv)} \Longrightarrow \makebox{(i)}$ Since $f$ is Riesz, by Lemma \ref{7.1}, $\sigma_{BF}(f)\subseteq \{0\}$. Hence $\sigma_{BF}(f)= \emptyset$, because $f$ is B-Fredholm by hypothesis. Now Corollary \ref{4.4} ensures that $f$ is algebraic. As in the proof of $\makebox{(ii)} \Longrightarrow \makebox{(i)}$ in Theorem \ref{5.5}, there exist orthogonal idempotents $p_{1},p_{2},\cdots,p_{n}$, commuting with $f$, such that $1=\sum\limits_{i=1}^{n} p_{i}$ and $(f-\lambda_{i})^{k_{i}}p_{i}=0$ for every $1 \leq i\leq n$. In order to complete the proof, it remains to show that $p_{i}$ lies in the socle of $\mathcal{A}$ when $\lambda_{i} \neq 0$.

Clearly $fp_{i}$ is a Riesz element in $p_{i}\mathcal{A}p_{i}$. In particular, $(f-\lambda_{i})p_{i}$ is a Fredholm element in $p_{i}\mathcal{A}p_{i}$.
Let $\phi:p_{i}\mathcal{A}p_{i}\rightarrow p_{i}\mathcal{A}p_{i}/ p_{i}\mathrm{soc}(A)p_{i}$ be the canonical quotient homomorphism. Hence, there is an element
$g_{i}$ in $p_{i}\mathcal{A}p_{i}$ such that
 $$\phi((f-\lambda_{i})p_{i})\phi(g_{i})=\phi(g_{i})\phi((f-\lambda_{i})p_{i})=\phi(p_{i}).$$
 Therefore, $$0=\phi((f-\lambda_{i})^{k_{i}}p_{i})\phi(g_{i}^{k_{i}})=\phi(p_{i}),$$
which is equivalent to $p_{i} \in p_{i}\mathrm{soc}(\mathcal{A})p_{i}\subseteq \mathrm{soc}(\mathcal{A})$.
\end{proof}







\section*{Acknowledgements}
This work has been supported by National Natural Science Foundation of China (Grant No. 11971108), Natural Science Foundation of Fujian Province (Grant No. 2020J01569) and Science and Technology Innovation Fund of Fujian Agriculture and Forestry University (Grant No. KFB23156).



\end{document}